\documentclass[10pt, a4paper]{amsart}  
\usepackage{amsmath,amssymb,amsfonts,hyperref}
\usepackage{graphicx}

\title{A quantitative version of Steinhaus' theorem for compact,
  connected, rank-one symmetric spaces}

\author{Fernando M\'ario de Oliveira Filho}
\address{F.M.~de Oliveira Filho, Institute of Mathematics, FU Berlin,
  Arnimallee~2, 14195 Berlin, Germany}
\email{fmario@math.fu-berlin.de}

\author{Frank Vallentin} 
\address{F.~Vallentin, Delft Institute of Applied Mathematics,
  Technical University of Delft, P.O.~Box 5031, 2600 GA Delft, The
  Netherlands}
\email{f.vallentin@tudelft.nl}

\thanks{The first author was partially supported by CAPES/Brazil under
  grant BEX 2421/04-6, and the research was partially carried out at the
  Centrum Wiskunde \& Informatica, The Netherlands. The second author
  is supported by Vidi grant 639.032.917 from the Dutch Organization
  for Scientific Research (NWO)}


\date{November 7, 2012}

\newcommand{\defi}[1]{\textit{#1}}

\newcommand{\R}{\mathbb{R}}

\newcommand{\Z}{\mathbb{Z}}
\newcommand{\C}{\mathbb{C}}

\newcommand{\Hu}{\mathbb{H}}
\newcommand{\Oc}{\mathbb{O}}

\newtheorem{defin}{Definition}[section]

\newtheorem{theorem}[defin]{Theorem}

\newtheorem{lemma}[defin]{Lemma}

\newtheorem{question}{Question}

\newcommand{\dens}{\overline{\delta}}

\newcommand{\RP}{{\R}\mathrm{P}}
\newcommand{\CP}{{\C}\mathrm{P}}
\newcommand{\HP}{{\Hu}\mathrm{P}}
\newcommand{\OP}{{\Oc}\mathrm{P}}

\def\Pab{P^{(\alpha,\beta)}}
\def\Pabm(#1,#2){P^{(\alpha #1,\beta #2)}}

\def\bPabm(#1,#2){\overline{P}^{(\alpha #1,\beta #2)}}

\def\tab[#1,#2]{t^{(\alpha,\beta)}_{#1,#2}}
\def\tabm(#1,#2)[#3,#4]{t^{(\alpha #1,\beta #2)}_{#3, #4}}

\def\lab{l^{(\alpha,\beta)}}

\def\breathe#1{\vadjust{\vskip#1\jot}}

\begin{document}

\begin{abstract}
  Let~$d_1$, $d_2$, \dots\ be a sequence of positive numbers that
  converges to zero. A generalization of Steinhaus' theorem due to
  Weil implies that, if a subset of a homogeneous Riemannian manifold
  has no pair of points at distances~$d_1$, $d_2$, \dots\ from each
  other, then it has to have measure zero. We present a quantitative
  version of this result for compact, connected, rank-one symmetric
  spaces, by showing how to choose distances so that the measure of a
  subset not containing pairs of points at these distances decays
  exponentially in the number of distances.
\end{abstract}

\maketitle

\markboth{F.M.~de Oliveira Filho and F.~Vallentin}{A quantitative version of Steinhaus' theorem}

\section{Introduction}
\label{sec:introduction}

\subsection{Steinhaus' theorem}

A well-known fact in measure theory is Steinhaus' theorem
(cf.~Steinhaus~\cite{Steinhaus}; see also Halmos~\cite{Halmos},
page~68). It states that if~$A$ is a Lebesgue measurable subset
of~$\R^n$ that has positive measure, then the origin lies in the
interior of the \defi{difference set}
\[
A - A = \{\, x - y : x, y \in A\,\}.
\]
Let~$H$ be a set of positive real numbers having
infimum~$0$. Steinhaus' theorem implies that the measure of a subset
of~$\R^n$ has to be zero, if it \defi{avoids} all the distances
in~$H$, i.e., if it does not contain any pair of points the distance
between which is in~$H$.

The \defi{upper density} (or simply \defi{density}) of a measurable
set~$A \subseteq \R^n$ is defined as
\[
\dens(A) = \limsup_{T \to \infty} \frac{\mu(A \cap [-T, T]^n)}{\mu([-T,T]^n)},
\]
where~$[-T,T]^n$ denotes the regular cube in~$\R^n$ with side~$2T$
centered at the origin and where~$\mu$ denotes the Lebesgue measure
in~$\R^n$. For a given set~$H$ of positive real numbers we let
\[
m_H(\R^n) = \sup\{\, \dens(A) : \text{$A \subseteq \R^n$ is measurable
and avoids all distances in~$H$}\,\}.
\]
In words,~$m_H(\R^n)$ is the maximum density a measurable subset
of~$\R^n$ can have if it avoids all distances in~$H$.  Steinhaus'
theorem implies that~$m_H(\R^n) = 0$ if~$\inf H = 0$.

Sz\'ekely~\cite{Szekely} asked whether the following continuity
statement applies to Steinhaus' theorem:

\begin{question}
\label{q:limit}
Let~$d_1$, $d_2$, \dots\ be a sequence of positive real numbers that
converges to zero. Is it true that
\[
\lim_{N \to \infty} m_{\{d_1, \ldots, d_N\}}(\R^n) = 0\text{?}
\]
\end{question}

The following example of Falconer~\cite{Falconer} shows that
Question~\ref{q:limit} has a negative answer in dimension one: We
consider the set
\[
A = \bigcup_{k \in \Z} (2k, 2k + 1).
\]
Set~$A$ has density~$1/2$ and avoids all odd distances. The scaled
set~$3^{-N} A$ also has density~$1/2$, and in particular it avoids
distances~$1$, $3^{-1}$, \dots,~$3^{-N}$. So the statement of
Question~\ref{q:limit} fails for the zero-convergent sequence~$1$,
$3^{-1}$, $3^{-2}$,~\dots\ when~$n=1$. It turns out that the real line
is an exceptional case: Falconer~\cite{Falconer} showed that the
answer to Question~\ref{q:limit} is affirmative for~$n \geq 2$.

\subsection{Quantitative statement}

Now we aim at transforming Question~\ref{q:limit} into a question in
hard analysis: Is it possible to turn Question~\ref{q:limit} into a
quantitative statement? Or, can one say anything about the rate of
convergence? So we are led to ask:

\begin{question}
\label{q:rate}
How fast does~$m_{\{d_1,\ldots,d_N\}}(\R^n)$ converge to zero if one is
allowed to choose the distances~$d_1$, \dots,~$d_N$?
\end{question}

One answer was given by Bukh~\cite{Bukh} based on a version of
Szemer\'edi's regularity lemma for measurable functions (see also
Chapter~2.4 in Tao~\cite{Tao}). He showed that the density decays
exponentially, and gave the exact asymptotic value for the density, as
the number of distances goes to infinity.

Oliveira and Vallentin~\cite{OliveiraVallentin} prove a slightly
weaker result that however already shows exponential decay. They show
that, for all~$N \geq 1$, there is a number~$r = r(N)$ such that, if
the sequence~$d_1$, \dots,~$d_N$ satisfies
\[
\text{$d_1 / d_2 > r$, $d_2 / d_3 > r$, \dots,~$d_{N-1}/d_N > r$},
\]
then
\[
m_{\{d_1,\ldots,d_N\}}(\R^n) \leq 2^{-N}.
\]
The proof of this result is rather short and is based on linear
programming and harmonic analysis.

The simplest Riemannian manifolds are, next to the Euclidean
space~$\R^n$, the connected, compact, rank-one symmetric spaces. They
are also called continuous, compact, two-point homogeneous spaces
because the isometry group acts transitively on pairs of points at
distance~$d$ for every~$d > 0$. Such spaces (with more than one point)
were classified by Wang~\cite{Wang}, the complete list being: the unit
sphere~$S^{n-1}$, the real projective space~$\RP^{n-1}$, the complex
projective space~$\CP^{n-1}$, the quaternionic projective
space~$\HP^{n-1}$, and the octonionic projective plane~$\OP^2$. Here
observe that~$S^{n-1} \subseteq \R^n$ and that~$\RP^{n-1}$ is the set
of all lines passing through the origin of~$\R^n$; the
spaces~$\CP^{n-1}$, $\HP^{n-1}$, and~$\OP^{n-1}$ are defined
similarly. Notice also that the \textit{real dimension} of~$S^{n-1}$,
that is, its dimension as a real manifold, is~$n-1$. Similarly, the
real dimension of~$\RP^{n-1}$ is~$n-1$, and the real dimensions
of~$\CP^{n-1}$, $\HP^{n-1}$, and~$\OP^2$ are~$2(n-1)$, $4(n-1)$,
and~$16$, respectively.

In this paper we consider quantitative versions of Steinhaus' theorem
for these spaces, providing answers to the analogues of
Questions~\ref{q:limit} and~\ref{q:rate} for them. It turns out that,
as in the case of the Euclidean space, one is able to show exponential
decay when the real dimension of the space considered is at least~$2$,
and when the real dimension is~$1$ it is possible to show that the
answer to the analogue of Question~\ref{q:limit} is negative.

To make ideas precise, we start by considering a generalization of
Steinhaus' theorem to locally compact groups due to Weil
(cf.~Weil~\cite{Weil}, page~50). Let~$G$ be a locally compact group and let~$A$
be a measurable subset of~$G$ having positive measure (here we
consider the Haar measure for~$G$). Weil's result
then says that the set
\[
A A^{-1} = \{\, xy^{-1} : x, y \in A\,\}
\]
contains the identity in its interior. A short and elementary proof is
due to Stromberg~\cite{Stromberg}. As in the case of~$\R^n$, this
theorem implies that subsets of a homogeneous space avoiding a set of
distances with infimum zero need to have measure zero. More precisely,
let~$M$ be a Riemannian manifold with isometry group~$G$ acting
transitively on~$M$, i.e.,~$M$ is a homogeneous $G$-space whose
measure is induced by the Haar measure of~$G$. Let~$H$ be a set of
positive real numbers having infimum zero. Weil's result then says
that the measure of a subset of~$M$ has to be zero, if it avoids all
distances given in~$H$.

If~$M$ is compact we normalize the measure~$\mu$ on~$M$ so
that~$\mu(M) = 1$. For a set of distances~$H$ we let
\[
m_H(M) = \sup\{\, \mu(A) : \text{$A \subseteq M$ is measurable and
avoids all distances in~$H$}\,\}.
\]
Then by Weil's generalization of Steinhaus' theorem, we have
that~$m_H(M) = 0$ if~$\inf H = 0$.

The main result of this paper is the following:
\medskip

\noindent
\textbf{Main result.}\enspace\emph{Let~$M$ be a compact, connected,
  rank-one symmetric space of real dimension at least~$2$. If
  distances~$d_1$, \dots,~$d_N$ are given so that~$d_1 \gg d_2 \gg
  \cdots \gg d_N$, then~$m_{\{d_1,\ldots,d_N\}}(M) \leq 2^{-N}$.  }
\medskip

In the statement of the main result we are purposely vague: The essence of
the statement is that, as long as the distances~$d_1$, \dots,~$d_N$
are sufficiently spaced out, then the maximum density is exponentially
small. In Section~\ref{sec:dimensiontwo} (in
Theorem~\ref{thm:main}) we will give a precise
statement of the main result in which we specify how the
distances should space out in order for the conclusion of the theorem
to hold.

Notice that this theorem implies that, if~$d_1$, $d_2$, \dots\ is a
sequence of distances converging to zero and the real dimension of the
space is at least~$2$, then
\begin{equation}
\label{eq:m-dens-limit}
\lim_{N \to \infty} m_{\{d_1, \ldots, d_N\}}(M) = 0,
\end{equation}
thus giving a positive answer to the analogue of
Question~\ref{q:limit} for the space~$M$. We will show in
Section~\ref{sec:dimensionone} that if the real dimension of~$M$
is~$1$, then there are zero-converging sequences~$d_1$, $d_2$, \dots\
of distances for which the limit in~\eqref{eq:m-dens-limit} is
nonzero, thus showing that Question~\ref{q:limit} has a negative
answer in dimension~$1$. So the big picture is similar to the one
of~$\R^n$ discussed earlier, and also the example that shows lack of decay
in dimension one is similar to the one we gave before for the real line.

\subsection{Outline of the paper}

In Section~\ref{sec:dimensionone} we show that the main result fails
for manifolds with real dimension~$1$.

We prove the main result in Section~\ref{sec:dimensiontwo}. For the
proof of this theorem we will use the generalization of the Lov\'asz
theta number for distance graphs defined over compact metric spaces
given by Bachoc, Nebe, Oliveira, and Vallentin~\cite{BNOV}. This
generalization will provide upper bounds for~$m_H(M)$ by means of the
solution of an infinite-dimensional semidefinite programming problem,
which reduces to an infinite-dimensional linear programming problem
when the Riemannian manifold~$M$ is a symmetric space. We recall the
necessary background for this approach in Section~\ref{sec:theta}.

Finally, in Section~\ref{sec:disc} we quickly discuss some questions
related to our approach.

\section{Counter-examples in dimension one}
\label{sec:dimensionone}

There are only two compact, connected, rank-one symmetric spaces~$M$
of real dimension one: The unit circle~$S^1$ and the real projective
line~$\RP^1$. For each of these two spaces we show that there is a
zero-convergent sequence of distances~$d_1$, $d_2$, \dots\ and
measurable subsets~$C_1$, $C_2$,~\dots\ of~$M$ all having measure at
least~$L$ for some positive constant~$L$, such that~$C_k$ avoids
distances~$d_1$, \dots,~$d_k$. This implies a negative answer to the
analogue of Question~\ref{q:limit} for the space~$M$.

We aim at presenting a single construction that works for both the
unit circle and the real projective line. To this end, for~$k = 1$,
$2$, \dots, write~$\theta_k = (\pi/2)/3^k$ and~$N_k = \lceil 3^k / 2
\rceil$. Let~$E_k = \{ e_{k,1}, \ldots, e_{k,{N_k}} \}$ be the set of
vectors in~$\R^2$ given as follows:

\begin{enumerate}
\item $e_{k,1} = (1, 0)$;

\item for~$i > 1$, vector~$e_{k,i}$ is equal to vector~$e_{k,{i-1}}$ rotated
  counterclockwise by an angle of~$2\theta_k$.
\end{enumerate}

Figure~\ref{fig:counter-ex} illustrates this construction. Notice that all
vectors in~$E_k$ belong to the unit circle~$S^1$ and moreover they all
lie in the nonnegative quadrant of~$\R^2$.

\begin{figure}[t]
\dimen0=.3\hsize
\def\picbox#1#2{#1}
\noindent
\hfill\picbox{\includegraphics{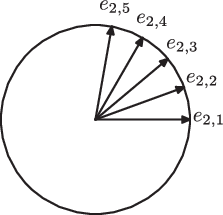}}{(a)}%
\hfill\picbox{\includegraphics{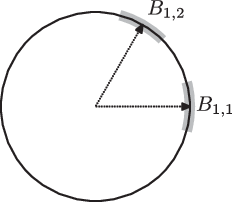}}{(b)}%
\hfill\picbox{\includegraphics{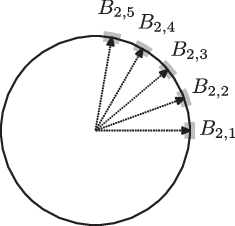}}{(c)}\hfill{}

\caption{On the left we have the vectors in set~$E_2$; the angle
  between any two consecutive vectors is~$\pi/9$. In the middle we
  show in gray the points in the set~$C_1 \subseteq S^1$ (that is, we
  take~$M = S^1$ in this example); notice that set~$C_1$ is the union
  of~$B_{1,1}$ and~$B_{1,2}$. Finally, on the right we show in gray
  the points in the set~$C_2 \subseteq S^1$; set~$C_2$ is the union
  of~$B_{2,1}$, \dots,~$B_{2,5}$.}
\label{fig:counter-ex}
\end{figure}

When~$M$ is the real projective line, we represent its elements by
unit vectors in~$\R^2$ by identifying antipodal vectors. So we can
see~$E_k$ as a subset of~$M$. Moreover, notice that by construction
the points in~$E_k$ are all distinct in~$M$.

Let~$d$ be the distance function of~$M$. If~$M = S^1$, then for~$x$,
$y \in S^1$ we have~$d(x, y) = \arccos(x \cdot y)$. If~$M = \RP^1$,
then~$d(x, y) = \arccos(2 (x \cdot y)^2 - 1)$. So write~$d_k =
d(e_{k,1}, e_{k,2})/2$. From the formula for the
distance we see that
\[
d(e_{k,i}, e_{k,j}) = 2 |i - j| d_k\qquad\text{for~$i$, $j = 1$,
  \dots,~$N_k$.}
\]
Moreover, one has~$d_k = 3 d_{k+1}$.

For~$i = 1$, \dots,~$N_k$ consider the set~$B_{k, i} = \{\, x \in M : d(x,
e_{k,i}) < d_k / 2\, \}$ and let
\[
C_k = \bigcup_{i = 1}^{N_k} B_{k,i}.
\]
Figure~\ref{fig:counter-ex} illustrates the construction of set~$C_k$.
We claim that set~$C_k$ avoids all distances of the form~$3^i d_k$
for~$i = 0$, \dots,~$k$.

Indeed, let~$x$, $y \in C_k$ and suppose~$x \in B_{k,i}$ and~$y \in
B_{k,j}$. Then
\[
d(x, y) \leq d(x, e_{k,i}) + d(e_{k,i}, e_{k,j}) + d(y, e_{k,j})
< d_k + 2 |i-j| d_k = (2 |i-j| + 1) d_k.
\]
On the other hand,
\[
d(e_{k,i}, e_{k,j}) \leq d(e_{k,i}, x) + d(x, y) + d(e_{k,j}, y) < d_k + d(x, y),
\]
whence it follows that~$d(x, y) > (2 |i-j| - 1) d_k$.

So we see that~$d(x, y)$ is never of the form~$(2l + 1) d_k$ for~$l =
0$, \dots,~$N_k - 1$. In particular, it is never of the form~$3^i d_k$
for~$i = 0$, \dots,~$k$, proving the claim.

Now consider the sequence~$d_1$, $d_2$, \dots\ of distances. This is a
zero-convergent sequence. Moreover, by our construction, and since
also~$d_k = 3 d_{k+1}$, set~$C_k$ avoids distances~$d_1$,
\dots,~$d_k$. 

Finally, notice that all sets~$B_{k, i}$ for fixed~$k$ and~$i = 1$,
\dots,~$N_k$ have the same positive measure. Moreover, since~$M$ has
real dimension~$1$, and since~$d_k = 3 d_{k+1}$, we have
\[
\mu(B_{k,i}) = 3 \mu(B_{k+1, i}).
\]
So we have
\[
\mu(C_k) = N_k 3^{1-k} \mu(B_{1, 1}) \geq (3/2)
\mu(B_{1,1}),
\]
and we see that the zero-convergent sequence of distances~$d_1$,
$d_2$, \dots\ and the sequence~$C_1$, $C_2$, \dots\ of subsets of~$M$
have the properties claimed in the beginning of the section.

\section{The theta number}
\label{sec:theta}
Let~$M$ be a compact Riemannian manifold with distance function~$d$
and~$G$ be its isometry group. Suppose~$M$ is a homogeneous $G$-space
and let~$\mu$ be the measure on~$M$ induced by the Haar measure
on~$G$, so that~$\mu$ is invariant under the action of~$G$. Moreover,
normalize~$\mu$ so that~$\mu(M) = 1$.

A \defi{kernel} is a continuous function~$K\colon M \times M \to
\R$. We say that~$K$ is \defi{positive} if the matrix~$\bigl(K(x_i,
x_j)\bigr)_{i,j=1}^N$ is positive semidefinite for all~$N$ and all
choices of points~$x_1$, \dots,~$x_N \in M$. We say that~$K$ is
\defi{invariant} if
\[
K(\sigma x, \sigma y) = K(x, y)\qquad
\text{for all~$\sigma \in G$ and~$x$, $y \in M$}.
\]

Now, given positive distances~$d_1$, \dots,~$d_N$, consider the
optimization problem
\begin{equation}
\label{opt:theta-primal}
\vcenter{\halign{\hfil#\enskip&#\hfil\cr
$\vartheta(M, \{d_1, \ldots, d_N\}) = \sup$&$\int_M \int_M K(x, y)\,
d\mu(x) d\mu(y)$\breathe1\cr
&$\int_M K(x, x)\, d\mu(x) = 1$,\breathe1\cr
&$K(x, y) = 0$\quad if~$d(x, y) \in \{ d_1, \ldots, d_N \}$,\cr
&$K\colon M \times M \to \R$ is a positive\cr
&\qquad and invariant kernel,\cr
}}
\end{equation}
which is an infinite-dimensional semidefinite programming
problem. This is a continuous version of the \textit{Lov\'asz theta
  number}, a graph parameter introduced by Lov\'asz~\cite{Lovasz79} to
upper bound the stability number of a finite graph.

Bachoc, Nebe, Oliveira, and Vallentin~\cite{BNOV} proved that
\[
\vartheta(M, \{ d_1, \ldots, d_N \}) \geq m_{\{d_1, \ldots, d_N\}}(M),
\]
so that by solving problem~\eqref{opt:theta-primal} one obtains an
upper bound for~$m_H(M)$.

Now, when~$M$ is a compact, connected, rank-one symmetric space,
one may decompose a positive and invariant kernel~$K\colon M \times M
\to \R$ in terms of Jacobi polynomials, as stated in the theorem below
(cf.~Askey~\cite{Askey}, page~65). We denote the Jacobi polynomial of
degree~$k$ with parameters~$\alpha$ and~$\beta$ by~$\Pab_k$, and we
normalize it so that~$\Pab_k(1) = 1$. Notice that this is not the
normalization commonly found in the literature on Jacobi polynomials
(for background on Jacobi polynomials, see e.g.~the book by
Szeg\"o~\cite{Szego}).

\begin{theorem}
\label{thm:decomp}
Let~$M$ be a compact, connected, rank-one symmetric space. A
kernel~$K\colon M \times M \to \R$ is positive and invariant if and
only if there are numbers~$\alpha$ and~$\beta$ such that
\begin{equation}
\label{eq:k-dec}
K(x, y) = \sum_{k = 0}^\infty f_k \Pab_k(\cos d(x, y))
\end{equation}
for some nonnegative numbers~$f_0$, $f_1$, \dots\ such
that~$\sum_{k=0}^\infty f_k$ converges, in which case the series
in~\eqref{eq:k-dec} converges absolutely and uniformly over~$M \times
M$.
\end{theorem}

The parameters~$\alpha$ and~$\beta$ depend on the space~$M$; they are
summarized in Table~\ref{tab:alpha-beta}.

\begin{table}[ht]
\begin{center}
\begin{tabular}{llcccc}
&&{\sl Real}&&\\
\multicolumn{2}{c}{{\sl Space}}&{\sl dimension}&$\alpha$&$\beta$\\[1pt]
\noalign{\hrule\vskip2pt}
Unit sphere&$S^{n-1}$&$n-1$&$(n-3)/2$&$(n-3)/2$\\
Real projective space&$\RP^{n-1}$&$n-1$&$(n-3)/2$&$-1/2$\\
Complex projective space&$\CP^{n-1}$&$2(n-1)$&$n-2$&0\\
Quaternionic projective space&$\HP^{n-1}$&$4(n-1)$&$2n-3$&1\\
Octonionic projective plane&$\OP^2$&16&7&3\\
\hline
\end{tabular}
\end{center}
\bigskip

\caption{Values of~$\alpha$ and~$\beta$ for Theorem~\ref{thm:decomp}
  depending on the space considered.}
\label{tab:alpha-beta}
\end{table}

Using Theorem~\ref{thm:decomp}, one may rewrite the optimization
problem~\eqref{opt:theta-primal}, obtaining that~$\vartheta(M, \{d_1,
\ldots, d_N\})$ is equal to the optimal value of the following
infinite-dimensional linear programming problem in variables~$f_k$:
\begin{equation}
\label{opt:theta-dec-primal}
\vcenter{\halign{\hfil#\enskip&#\hfil&\quad#\hfil\cr
$\sup$&$f_0$\breathe1\cr
&$\sum_{k=0}^\infty f_k = 1$,\breathe1\cr
&$\sum_{k=0}^\infty f_k \Pab_k(\cos d_i) = 0$&for~$i = 1$, \dots,~$N$,\breathe1\cr
&$f_k \geq 0$&for~$k = 0$, $1$, \dots.\cr
}}
\end{equation}
In this reformulation, we used the fact that~$K$ is invariant and~$M$
is homogeneous, so that all the diagonal entries of~$K$ are the same,
and therefore equal to~$1$. Moreover, we also use our
normalization~$\Pab_k(1) = 1$ for the Jacobi polynomials, and the
orthogonality property of the Jacobi polynomials.

A possible dual for problem~\eqref{opt:theta-dec-primal} is the
following linear programming problem on variables~$z_0$, $z_1$,
\dots,~$z_N$ with infinitely many constraints:
\begin{equation}
\label{opt:theta-dec-dual}
\vcenter{\halign{\hfil#\enskip&#\hfil&\quad#\hfil\cr
$\inf$&$z_0$\cr
&$z_0 + z_1 + \cdots + z_N \geq 1$,\cr
&$z_0 + z_1 \Pab_k(\cos d_1) + \cdots + z_N \Pab_k(\cos d_N) \geq
0$&for~$k = 1$, $2$, \dots.\cr
}}
\end{equation}

It is easy to show that weak duality holds
between~\eqref{opt:theta-dec-primal} and~\eqref{opt:theta-dec-dual},
that is, that any feasible solution of~\eqref{opt:theta-dec-dual}
provides an upper bound for the optimal value
of~\eqref{opt:theta-dec-primal}. Indeed, if~$f_0$, $f_1$, \dots\ is a
feasible solution of~\eqref{opt:theta-dec-primal} and~$(z_0, z_1,
\ldots, z_N)$ is a feasible solution of~\eqref{opt:theta-dec-dual},
then
\[
f_0 \leq \sum_{k=0}^\infty f_k (z_0 + z_1 \Pab_k(\cos d_1) + \cdots +
z_N \Pab_k(\cos d_N)) = z_0,
\]
as we wanted.

The optimization problem~\eqref{opt:theta-dec-dual} is the tool we use
to prove the main result. Notice that, because of weak
duality, it is not necessary to solve~\eqref{opt:theta-dec-dual} to
optimality in order to get an upper bound for~$m_H(M)$, any feasible
solution will provide such a bound.

\section{A proof of the main theorem}
\label{sec:dimensiontwo}

In this section we make the statement of our main result 
precise by giving a condition that, if satisfied by the distances,
implies exponential density decay. More specifically we prove the
following theorem.

\begin{theorem}
\label{thm:main}
Let~$M$ be a compact, connected, rank-one symmetric space of real
dimension at least~$2$ and let~$N \geq 1$. There is a number~$0 < d_0
< \pi / 2$ and a function~$r\colon (0, d_0) \to (0, d_0)$, depending
on~$N$, such that if numbers~$d_1 > \cdots > d_N$ are picked from the
interval~$(0, d_0)$ so as to satisfy
\begin{equation}
\label{eq:dist-cond}
d_{i+1} \leq r(d_i)\qquad\text{for~$i = 1$, \dots,~$N-1$},
\end{equation}
then~$m_{\{d_1, \ldots, d_N\}}(M) \leq 2^{-N}$.
\end{theorem}

Recall that we denote by~$\Pab_k$ the Jacobi polynomial of degree~$k$
and parameters~$\alpha$ and~$\beta$, and that we normalize it so
that~$\Pab_k(1) = 1$. This normalization differs from the one normally
adopted in the literature about Jacobi polynomials, where one usually
sets
\[
\Pab_k(1) = {k + \alpha\choose k}.
\]

One important property of our normalization is the following:
\begin{equation}
\label{eq:jac-zero}
\dimen0=\hsize \advance\dimen0 by-1.5cm
\vcenter{\hsize=\dimen0%
\noindent
For~$\alpha \geq 0$ and all intervals~$[a, b] \subseteq (-1, 1)$, we
have that~$\Pab_k(u) \to 0$ as~$k \to \infty$ uniformly in the
interval~$[a, b]$.}
\end{equation}
This follows from the asymptotic formula for the Jacobi polynomials
given in Theorem~8.21.8 of Szeg\"o~\cite{Szego}, together with the
fact that for~$\alpha \geq 0$ we have~${k+\alpha\choose k} \geq
1$. This fact will play an important role in the rest of this section.

For~$\alpha$, $\beta > -1$ and~$-1 < t < 1$, let
\[
\lab(t) = \inf\{\, \Pab_k(t) : \text{$k = 0$, $1$, \dots}\,\}.
\]
Note that~$\lab(t) < 0$ for all~$-1 < t < 1$. This follows from a
simple application of the interlacing property for Jacobi
polynomials (cf.~Theorem~3.3.2 in Szeg\"o~\cite{Szego}), or
alternatively from the asymptotic formula for Jacobi polynomials
(cf.~Theorem~8.21.8 in Szeg\"o~\cite{Szego}).

In our proof of Theorem~\ref{thm:main} we will use the following
rather technical result about the behavior of~$\lab(t)$.

\begin{lemma}
\label{lem:jac-bounds}
For all~$\alpha \geq 0$ and~$\beta \geq -1/2$ with~$\alpha \geq
\beta$, there is a number~$t_0$ with~$0 < t_0 < 1$ such that~$\lab(t)
\geq -1/2$ for all~$t_0 < t < 1$.
\end{lemma}

Before proving this lemma, we first prove
Theorem~\ref{thm:main}.

\begin{proof}[Proof of Theorem~\ref{thm:main}]
Let~$t_0$ be given as in Lemma~\ref{lem:jac-bounds} and set~$d_0 =
\arccos t_0$. We now analyze the linear programming
problem~\eqref{opt:theta-dec-dual} for distances~$d_1$, \dots,~$d_N$,
with~$\alpha$ and~$\beta$ given as in Table~\ref{tab:alpha-beta}
according to the space~$M$ considered. Here it is important to observe
that, since the real dimension of~$M$ is at least~$2$, we have~$\alpha
\geq 0$.

If~$N = 1$, then since~$\lab(t) \geq -1/2$ for all~$t_0 < t < 1$, it
is easy to see from~\eqref{opt:theta-dec-dual} that~$m_{\{d_1\}}(M)
\leq 1/2$ for all~$d_0 > d_1 > 0$, just by setting~$z_0 = 1/2$ and~$z_1 = 1$.

So suppose~$N > 1$. Let
\[
\lambda = |\inf\{\, \lab(u) : t_0 < u < 1\,\}|.
\]
By the choice of~$t_0$, and since~$\lab(t) < 0$ for~$-1 < t < 1$, we
must have~$\lambda \leq 1/2$.  Write~$\varepsilon = \lambda^{N+1}/((1
- \lambda)(N-1))$.

We now define function~$r$. So let~$d \in (0, d_0)$ be given. We show
how to pick a number~$r(d) \in (0, d_0)$ having the following
property: 
\begin{equation}
\label{eq:rd-prop}
\dimen0=\hsize \advance\dimen0 by-1.5cm
\vcenter{\hsize=\dimen0%
\noindent
For every~$0 < s \leq r(d)$, if~$\Pab_k(\cos s) \leq 1 - \varepsilon$,
then~$|\Pab_k(\cos d')| < \varepsilon$ for all~$d \leq d' \leq \pi / 2$.}
\end{equation}

To prove that we may pick such a number~$r(d)$, we use the fact
that~$\Pab_k(t) \to 0$ as~$k \to \infty$ uniformly in the
interval~$[0, \cos d]$, as observed in~\eqref{eq:jac-zero}. So there
is a~$k_0$ such that~$|\Pab_k(t)| < \varepsilon$ for all~$0 \leq t
\leq \cos d$ and~$k > k_0$. Now, since each~$\Pab_k$ is continuous
and~$\Pab_k(1) = 1$, we may pick a number~$u_0$ such that~$\Pab_k(u) >
1 - \varepsilon$ for all~$u_0 \leq u \leq 1$ and~$k \leq k_0$. But
then we may set~$r(d) = \arccos u_0$ and~\eqref{eq:rd-prop} is
satisfied.

Suppose now that numbers~$d_1$, \dots,~$d_N \in (0, d_0)$
with~$d_1 > \cdots > d_N$ are given which
satisfy~\eqref{eq:dist-cond}. We claim that, for~$1 \leq j \leq N$,
\begin{equation}
\label{eq:decay-main-claim}
\sum_{i=1}^j \lambda^{i-1} \Pab_k(\cos d_i) \geq -\lambda^j -
\varepsilon(j-1)\qquad\text{for all~$k \geq 0$}.
\end{equation}

Before proving the claim, let us show how to apply it in order to
prove the theorem. Taking~$j = N$ we see that
\[
\sum_{i=1}^N \lambda^{i-1} \Pab_k(\cos d_i) \geq -\lambda^N -
\varepsilon(N-1)\qquad\text{for all~$k \geq 0$}.
\]
So, letting~$S = 1 + \lambda + \cdots + \lambda^N + \varepsilon
(N-1)$, we may set
\[
z_0 = \frac{\lambda^N + \varepsilon(N-1)}{S}\qquad\text{and}\qquad
z_i = \frac{\lambda^{i-1}}{S}\quad\text{for~$i = 1$, \dots,~$N$}
\]
and check that this is a feasible solution
of~\eqref{opt:theta-dec-dual}. But then, since~$\lambda \leq 1/2$ and
also~$\varepsilon = \lambda^{N+1}/((1 - \lambda)(N-1))$, from the
weak duality relation between~\eqref{opt:theta-dec-primal}
and~\eqref{opt:theta-dec-dual} we get that
\[
\begin{split}
m_{\{d_1, \ldots, d_N\}}(M)&\leq \frac{\lambda^N + \varepsilon(N-1)}{1 +
  \lambda + \cdots + \lambda^N + \varepsilon(N-1)}\\
&= \frac{\lambda^N + \lambda^{N+1} / (1 - \lambda)}{(1 - \lambda^{N+1})/(1
  - \lambda) + \lambda^{N+1}/(1 - \lambda)}\\
&= \lambda^N (1 - \lambda) + \lambda^{N+1}\\
&\leq 2^{-N},
\end{split}
\]
where we use that~$\lambda(1 - \lambda) \leq 1/4$, and the theorem
follows.

We finish by proving~\eqref{eq:decay-main-claim}. For~$j = 1$, the
statement is true by the definition of~$\lambda$. Now suppose the
statement is true for some~$1 \leq j < N$; we show that it is also
true for~$j+1$. To this end, let~$k \geq 0$ be an
integer. If~$\Pab_k(\cos d_{j+1}) > 1 - \varepsilon$, then by using
the induction hypothesis and since~$\lambda \leq 1$ we get
\[
\begin{split}
\sum_{i=1}^{j+1} \lambda^{i-1} \Pab_k(\cos d_i)&= \lambda^j
\Pab_k(\cos d_{j+1}) + \sum_{i=1}^j \lambda^{i-1} \Pab_k(\cos d_i)\\
&\geq \lambda^j (1 - \varepsilon) - \lambda^j - \varepsilon(j-1)\\
&\geq -\varepsilon j.
\end{split}
\]

If, on the other hand,~$\Pab_k(\cos d_{j+1}) \leq 1 - \varepsilon$, we
know from the choice of the~$d_i$ and from~\eqref{eq:rd-prop}
that~$|\Pab_k(\cos d_i)| < \varepsilon$ for~$i = 1$, \dots,~$j$. But
then we have
\[
\begin{split}
\sum_{i=1}^{j+1} \lambda^{i-1} \Pab_k(\cos d_i)&= \lambda^j
\Pab_k(\cos d_{j+1}) + \sum_{i=1}^j \lambda^{i-1} \Pab_k(\cos d_i)\\
&\geq -\lambda^{j+1} - \varepsilon j,
\end{split}
\]
proving~\eqref{eq:decay-main-claim}.
\end{proof}

Let~$\tab[k,1] < \cdots < \tab[k,k]$ be the zeros of~$\Pab_k$ in
ascending order. In the proof of Lemma~\ref{lem:jac-bounds} we will
use the following two facts proven by Wong and
Zhang~\cite{WongZ}. Let~$\alpha \geq 0$ and~$\beta \geq -1/2$. The
first fact is that
\begin{equation}
\label{eq:jacobi-limit}
\lim_{k \to \infty} \Pab_k(\tabm(+1,+1)[k-1,k-1]) = 
\Gamma(\alpha+1) \Bigl(\frac{2}{j_{\alpha+1}}\Bigr)^\alpha
J_\alpha(j_{\alpha+1}),
\end{equation}
where~$J_\nu$ is the Bessel function of the first kind of order~$\nu$
and~$j_\nu$ is its first positive zero. Note that the right-hand side
of the expression above is negative for all~$\alpha \geq 0$.

The second fact we use is that, for all large enough~$k$,
\begin{equation}
\label{eq:jac-last-extremum}
\Pab_{k+1}(\tabm(+1,+1)[k,k]) > \Pab_k(\tabm(+1,+1)[k-1,k-1]).
\end{equation}

Both these facts follow from~(1.6) in Wong and Zhang~\cite{WongZ}. In
their paper, however, formula~(1.6) is proven under the condition
that~$\alpha > \beta > -1/2$. One may verify however that the same
proof holds when~$\alpha \geq 0$ and~$\beta \geq -1/2$. Indeed, the
assumption that~$\alpha > \beta > -1/2$ is made by Wong and Zhang so
that formulas~(2.3) and~(2.6) might be used. Both formulas appear in a
paper by Frenzen and Wong~\cite{FrenzenW} (formula~(2.3) is the Main
Theorem in that paper and formula~(2.6) is Corollary~2), and there
they are stated under more general conditions that include all cases
in which~$\alpha \geq 0$ and~$\beta \geq -1/2$.

\begin{proof}[Proof of Lemma~\ref{lem:jac-bounds}]
In our proof we use the following claim; a stronger version of this
claim is stated in~(6.4.24) of Andrews, Askey, and Roy~\cite{AAR} for
the case of the ultraspherical polynomials.\medskip

\noindent
{\it Claim A.\enspace For every~$\alpha \geq 0$ and~$\beta > -1$
  with~$\alpha \geq \beta$, there is a number~$k_0 \geq 2$ such that for all~$k
  \geq k_0$ we have}
\[
\Pab_k(\tabm(+1,+1)[k-1,k-1]) = \min\{\, \Pab_k(u) : 0 \leq u \leq 1\,\}.
\]
{\it Proof.}\enspace Consider the function
\begin{equation}
\label{eq:g-jacobi}
g(u) = (\Pab_k(u))^2 + \frac{1 - u^2}{k(k+\alpha+\beta+1)}
\biggl(\frac{d\Pab_k(u)}{du}\biggr)^2.
\end{equation}
We use the identity
\[
(1 - u^2) \frac{d^2\Pab_k(u)}{du^2} + (\beta - \alpha - (\alpha +
\beta + 2)u) \frac{d\Pab_k(u)}{du} + k(k + \alpha + \beta + 1)
\Pab_k(u) = 0
\]
(cf.~(6.3.9) in Andrews, Askey, and Roy~\cite{AAR}) to obtain
\[
g'(u) = \frac{2 ((\alpha + \beta + 1) u + \alpha - \beta)}{k(k +
  \alpha + \beta + 1)} \biggl(\frac{d \Pab_k(u)}{du}\biggr)^2.
\]

From this it is at once clear that, since~$\alpha \geq 0$,~$\beta >
-1$, and~$\alpha \geq \beta$, we will have~$g'(u) \geq 0$ for
all~$0 \leq u \leq 1$. This means that~$g$ is nondecreasing in the
interval~$(0, 1)$. So, if for some~$k$ and~$i$ we
have~$\tabm(+1,+1)[k-1,i] > 0$, then since~$\tabm(+1,+1)[k-1,i]
\leq \tabm(+1,+1)[k-1,k-1] < 1$ we also have
\[
g(\tabm(+1,+1)[k-1,i]) \leq g(\tabm(+1,+1)[k-1,k-1]),
\]
and since the zeros of~$\Pabm(+1,+1)_{k-1}$ correspond to the extrema
of~$\Pab_k$ (cf.~(6.3.8) in Andrews, Askey, and Roy~\cite{AAR}),
with~\eqref{eq:g-jacobi} this amounts to
\begin{equation}
\label{eq:jac-dec-extrema}
(\Pab_k(\tabm(+1,+1)[k-1,i]))^2 \leq (\Pab_k(\tabm(+1,+1)[k-1,k-1]))^2.
\end{equation}

Now, we know that there is a~$k_0$ such that~$\tabm(+1,+1)[k-1,k-1] >
0$ for all~$k \geq k_0$ (cf.~Theorem~6.1.1 of
Szeg\"o~\cite{Szego}). Suppose also~$k_0$ is large enough so
that~$\Pab_k(\tabm(+1,+1)[k-1,k-1])$ is close to the right-hand side
of~\eqref{eq:jacobi-limit} for all~$k \geq
k_0$. From~\eqref{eq:jac-zero}, we may also assume that~$k_0$ is such
that~$|\Pab_k(u)|$ is close enough to zero in some
interval~$[-\varepsilon, 0]$ for some~$0 < \varepsilon < 1$ chosen
arbitrarily.  But then it is clear that the minimum of~$\Pab_k(u)$ in
the interval~$[0, 1]$ is achieved in the interior of the interval, and
then since the~$\tabm(+1,+1)[k-1,i]$ are the extrema of~$\Pab_k$, and
since~$\tabm(+1,+1)[k-1,k-1]$ is a local minimum of~$\Pab_k$, together
with~\eqref{eq:jac-dec-extrema} we are done.\hfill$\lhd$\medskip

\noindent {\it Claim B.\enspace For~$\alpha \geq 0$ the expression on
the right-hand side of~\eqref{eq:jacobi-limit} is always at
least~$-0.45$.}\medskip

\noindent
{\it Proof.}\enspace For~$\alpha \geq 0$ and~$t > 0$ write
\[
\Omega_\alpha(t) = \Gamma(\alpha+1) \Bigl(\frac{2}{t}\Bigr)^\alpha
J_\alpha(t).
\]
The global minimum of~$\Omega_\alpha$ is
attained at~$j_{\alpha+1}$, the first positive zero of~$J_{\alpha+1}$,
and it is negative (cf.~(4.6.2) in Andrews, Askey, and
Roy~\cite{AAR} and Section~15.31 of Watson~\cite{Watson}).

We first show that~$\Omega_\alpha(t) \geq -0.45$ for all~$0 \leq
\alpha < 1$. Indeed, the zero~$j_{\alpha+1}$ is increasing as a
function of~$\alpha$ (cf.~(2) in Section~15.6 of Watson~\cite{Watson})
and one may verify numerically that~$j_1 = 3.8317\ldots \geq 2$. So
from the definition of~$\Omega_\alpha$ and since the global minimum
of~$\Omega_\alpha$ is attained at~$j_{\alpha+1}$ we see that
\[
\Omega_\alpha(t) \geq J_\alpha(j_{\alpha+1})
\]
whenever~$0 \leq \alpha < 1$. Landau~\cite{Landau} has shown
that~$J_\alpha(j_{\alpha+1})$ is increasing as a function
of~$\alpha$. But then for~$0 \leq \alpha < 1$ we have
\[
\Omega_\alpha(t) \geq J_0(j_1) = -0.402759\ldots \geq -0.45,
\]
as we wanted.

Suppose now~$\alpha \geq 1$. We use the formula
\[
J_{\alpha-1}(t) + J_{\alpha+1}(t) = \frac{2\alpha}{t} J_\alpha(t)
\]
(cf.~(4.6.5) in Andrews, Askey, and Roy~\cite{AAR}) to see that
\[
\begin{split}
\Omega_\alpha(t)&= \Gamma(\alpha+1) \Bigl(\frac{2}{t}\Bigr)^\alpha
\frac{t}{2\alpha}(J_{\alpha-1}(t) + J_{\alpha+1}(t))\\
&= \Omega_{\alpha-1}(t) + \Gamma(\alpha)
\Bigl(\frac{2}{t}\Bigr)^{\alpha-1} J_{\alpha+1}(t).
\end{split}
\]
Since the global minimum of~$\Omega_\alpha$ is attained
at~$j_{\alpha+1}$, this implies that the global minimum
of~$\Omega_\alpha$ is at least that of~$\Omega_{\alpha-1}$. Since we
have shown that~$\Omega_\alpha(t) \geq -0.45$ for all~$0 \leq \alpha <
1$, the result then follows.\hfill$\lhd$\medskip

Now we may prove the lemma. Let~$k_0$ be given as in
Claim~A. Suppose that~$k_0$ is large enough so
that~$\Pab_k(\tabm(+1,+1)[k-1,k-1])$ is close enough to the right-hand
side of~\eqref{eq:jacobi-limit} for all~$k \geq k_0$, in such a way
that from Claim~B we see that
\[
\Pab_k(\tabm(+1,+1)[k-1,k-1]) \geq -1/2\qquad\text{for all~$k \geq
  k_0$.}
\]

Suppose also that~$k_0$ is large enough so
that~\eqref{eq:jac-last-extremum} holds for all~$k \geq k_0$. Now
take some~$k \geq k_0$ such that~$t_0 = \tabm(+1,+1)[k-1,k-1] >
0$. We show that~$\lab(t) \geq -1/2$ for all~$t_0 < t < 1$.

So fix some~$t$ such that~$t_0 < t < 1$. We begin by showing that the
sequence
\begin{equation}
\label{eq:low-seq}
\text{$\Pab_0(t)$, $\Pab_1(t)$, \dots,~$\Pab_k(t)$}
\end{equation}
is decreasing. To this end we shall use the formula
\begin{equation}
\label{eq:jac-intersec}
\Pab_j(u) - \Pab_{j+1}(u) = \frac{2j + \alpha + \beta +
  2}{2(\alpha+1)} (1 - u) \Pabm(+1,)_j(u),
\end{equation}
which is adapted to our normalization of~$\Pab_j$ from~(6.4.20) in
Andrews, Askey, and Roy~\cite{AAR}.

Now, for~$j < k$ we have
\[
\tabm(+1,)[j,j] \leq \tabm(+1,)[k-1,k-1] < \tabm(+1,+1)[k-1,k-1] < t.
\]
Here, the first inequality comes from the interlacing property. The
second inequality is a consequence of Theorem~6.21.1 of
Szeg\"o~\cite{Szego}. So~$t$ lies to the right of the rightmost zero
of~$\Pabm(+1,)_j$ and hence~$\Pabm(+1,)_j(t) > 0$. So it is clear
from~\eqref{eq:jac-intersec} that~$\Pab_j(t) > \Pab_{j+1}(t)$, proving
that~\eqref{eq:low-seq} is decreasing. 

The fact that~\eqref{eq:low-seq} is decreasing, together with Claim~A,
implies that
\begin{equation}
\label{eq:low-deg-claim}
\Pab_j(t) > \Pab_k(t) \geq \Pab_k(t_0)
\end{equation}
for all~$j < k$. Now suppose~$j > k$. We then have that
\[
\Pab_j(t) \geq \Pab_j(\tabm(+1,+1)[j-1,j-1]) >
\Pab_k(\tabm(+1,+1)[k-1,k-1]) = \Pab_k(t_0),
\]
where the first inequality follows from Claim~A and the second
inequality follows from~\eqref{eq:jac-last-extremum}. Now, since we
have by construction that~$\Pab_k(t_0) \geq -1/2$, together
with~\eqref{eq:low-deg-claim} we see that for all~$j \neq k$ we
have~$\Pab_j(t) \geq \Pab_k(t_0) \geq -1/2$, and we are done.
\end{proof}

\section{Some questions}
\label{sec:disc}

Here are some questions that arise from our approach.
\begin{enumerate}
\item Up to now we only know spaces where the density decay is either
exponential or not present at all. Are there spaces for which we have only
polynomial density decay?

\item Curvature of the unit norm ball seems to play an important role. In
all the cases we know we have exponential density decay if and only if
the unit norm ball possesses positive curvature. Is this true in
general? An affirmative result in this direction is due to
Kolountzakis~\cite{Kolountzakis} who proved this for~$\mathbb{R}^n$
equipped with an arbitrary Minkowski norm.

\item It might be interesting to understand the limits of our approach
  using the Lov\'asz theta number. Is this method always able to detect
  (exponential) density decay? For finite graphs Alon and
  Kahale~\cite{AlonKahale} were able to show that one can approximate
  the stability number by the theta number. Our method can be seen as
  a generalization of the theta number to infinite graphs.
\end{enumerate}

\section{Acknowledgements}

We would like to thank Roderick Wong for his comments and observations
concerning the results of~\cite{WongZ}. We thank the referee for
valuable comments.

\end{document}